\documentclass[]{amsart}
\pdfoutput=1

\usepackage{latexsym}   
\usepackage{amssymb}    
\usepackage[colorlinks]{hyperref}
\usepackage[margin=1in]{geometry}
\usepackage[all]{xy}

\newtheorem{theorem}{Theorem}[section]
\newtheorem{lemma}[theorem]{Lemma}
\newtheorem{corollary}[theorem]{Corollary}
\newtheorem{proposition}[theorem]{Proposition}
\newtheorem{definition}[theorem]{Definition}

\theoremstyle{remark}
\newtheorem{example}[theorem]{Example}
\newtheorem{remark}[theorem]{Remark}

\newcommand{\Fp}{{\mathbf F}_p}
\newcommand{\Fq}{{\mathbf F}_q}
\newcommand{\BC}{{\mathbf C}}
\newcommand{\BF}{{\mathbf F}}
\newcommand{\BQ}{{\mathbf Q}}
\newcommand{\BR}{{\mathbf R}}
\newcommand{\BZ}{{\mathbf Z}}

\newcommand{\calC}{{\mathcal C}}
\newcommand{\calO}{{\mathcal O}}
\newcommand{\calS}{{\mathcal S}}
\newcommand{\calOp}{\calO^+}

\newcommand{\Aid}{{\mathfrak A}}
\newcommand{\Bid}{{\mathfrak B}}
\newcommand{\Cid}{{\mathfrak C}}
\newcommand{\did}{{\mathfrak d}}
\newcommand{\Did}{{\mathfrak D}}
\newcommand{\Eid}{{\mathfrak E}}
\newcommand{\fid}{{\mathfrak f}}
\newcommand{\Fid}{{\mathfrak F}}
\newcommand{\mideal}{{\mathfrak m}}
\newcommand{\pid}{{\mathfrak p}}
\newcommand{\Pid}{{\mathfrak P}}
\newcommand{\qid}{{\mathfrak q}}
\newcommand{\Qid}{{\mathfrak Q}}

\newcommand{\alphabar}{\bar{\alpha}}
\newcommand{\betabar}{\bar{\beta}}
\newcommand{\bbar}{\overline{b}}  
\newcommand{\kb}{\bar{k}}
\newcommand{\pibar}{\bar{\pi}}
\newcommand{\Aidbar}{\bar{\Aid}}
\newcommand{\Bidbar}{\bar{\Bid}}
\newcommand{\qidbar}{\bar{\qid}}
\newcommand{\ubar}{\bar{u}}
\newcommand{\xbar}{\bar{x}}
\newcommand{\zbar}{\bar{z}}

\newcommand{\Adual}{\hat{A}}
\newcommand{\Khat}{\hat{K}}
\newcommand{\Rhat}{\hat{R}}
\newcommand{\Zhat}{\hat{\BZ}}

\newcommand{\eps}{\varepsilon}

\newcommand*{\triplesim}{%
      \mathrel{\vcenter{\offinterlineskip
      \hbox{$\sim$}\vskip-.4ex\hbox{$\sim$}\vskip-.4ex\hbox{$\sim$}}}}

\DeclareMathOperator{\End}{End}
\DeclareMathOperator{\Inv}{Inv}
\DeclareMathOperator{\modstar}{mod^*}
\DeclareMathOperator{\reg}{reg}
\DeclareMathOperator{\Pic}{Pic}
\DeclareMathOperator{\Picplus}{\Pic^+}
\DeclareMathOperator{\Prin}{Prin}
\DeclareMathOperator{\Tr}{Tr}
\DeclareMathOperator{\USp}{USp}

\newcommand{\hm}{h^{-}}
\newcommand{\Kp}{K^{+}}
\newcommand{\Rp}{R^{+}}
\newcommand{\Up}{U^{+}}

\newcommand{\Vladut}{Vl\u adu\c t}

\newcommand{\col}{:}

\newcommand{\mybar}[1]{
  \mathchoice
  {#1\llap{$\overline{\phantom{\displaystyle\rm#1}}$}}
  {#1\llap{$\overline{\phantom{\textstyle\rm#1}}$}}
  {#1\llap{$\overline{\phantom{\scriptstyle\rm#1}}$}}
  {#1\llap{$\overline{\phantom{\scriptscriptstyle\rm#1}}$}}
}  
\renewcommand{\bar}{\mybar}
\renewcommand{\hat}{\widehat}

\makeatletter
\@namedef{subjclassname@2020}{%
  \textup{2020} Mathematics Subject Classification}
\makeatother

\begin{document}

\title[Principally polarized abelian varieties]
      {Variations in the distribution \\ of principally polarized abelian varieties\\ among isogeny classes}
\author{Everett W.~Howe}
\address{Independent mathematician, 
         San Diego, CA 92104, USA.}
\email{however@alumni.caltech.edu}
\urladdr{http://ewhowe.com}
\date{28 May 2020}




\begin{abstract}
We show that for a large class of rings $R$, the number of principally
polarized abelian varieties over a finite field in a given simple ordinary 
isogeny class and with endomorphism ring $R$ is equal either to~$0$, or to a 
ratio of class numbers associated to $R$, up to some small computable factors.
This class of rings includes the maximal order of the CM~field $K$ associated
to the isogeny class (for which the result was already known), as well as the
order $R$ generated over $\BZ$ by Frobenius and Verschiebung.

For this latter order, we can use results of Louboutin to estimate the 
appropriate ratio of class numbers in terms of the size of the base field
and the Frobenius angles of the isogeny class. The error terms in our estimates
are quite large, but the trigonometric terms in the estimate are suggestive: 
Combined with a result of \Vladut\ on the distribution of Frobenius angles 
of isogeny classes, they give a heuristic argument in support of the theorem
of Katz and Sarnak on the limiting distribution of the multiset of Frobenius
angles for principally polarized abelian varieties of a fixed dimension over
finite fields.
\end{abstract}

\maketitle


\section{Introduction}
\label{S:intro}

In this paper we consider the problem of estimating the number of isomorphism
classes of principally polarized abelian varieties $(A,\lambda)$ such that $A$
lies in a given isogeny class of simple ordinary abelian varieties over a finite
field. We approach this problem by subdividing isogeny classes into their 
\emph{strata}, which are the subsets of an isogeny class consisting of abelian 
varieties sharing the same endomorphism ring. (To avoid awkward locutions,
we will say that a principally polarized variety $(A,\lambda)$ lies in an
isogeny class $\calC$ or a stratum $\calS$ when $A$ lies in $\calC$ or $\calS$.)

Our main result concerns strata corresponding to endomorphism rings $R$ that are
\emph{convenient}. A convenient ring is an order in a CM~field with the 
properties that, first, $R$ is stable under complex conjugation; second, the
real subring $\Rp$ of $R$ is Gorenstein; and third, the trace dual of $R$ is
generated by its pure imaginary elements. (We explain these terms and present
results on convenient rings in Section~\ref{S:convenient}.) If $\calS$ is a 
stratum of an isogeny class corresponding to a convenient order~$R$, we can 
express the number of principally polarized varieties in $\calS$ in terms of the
sizes of the Picard group of $R$ and the narrow Picard group of the maximal real
sub-order $\Rp$ of~$R$; the definitions of these groups are also reviewed in
Section~\ref{S:convenient}.

\begin{theorem}
\label{T:classgroup}
Let $\calS$ be a stratum of an isogeny class of simple ordinary abelian
varieties over a finite field, corresponding to an endomorphism ring $R$. 
Suppose that $R$ is convenient and that the norm map $N_{\Pic}$ from the Picard
group of $R$ to the narrow Picard group of $\Rp$ is surjective. Let $U$ be the 
unit group of $R$ and let $\Up_{>0}$ be the group of totally positive units 
of~$\Rp$. Then the number of varieties $A\in\calS$ that have principal 
polarizations is equal to $\#\ker N_{\Pic}$, and each such $A$ has 
$[\Up_{> 0}\col N(U)]$ principal polarizations up to isomorphism, where $N$ is
the norm map from $R$ to~$\Rp$.
\end{theorem}

\begin{corollary}
\label{C:classgroup}
Under the hypotheses of Theorem~\textup{\ref{T:classgroup}}, the total number of
principally polarized varieties $(A,\lambda)$ in the stratum $\calS$, counted up
to isomorphism, is equal to 
\[\frac{1}{[N(U) : (\Up)^2]} \, \frac{\#\Pic R^{\phantom{+}}}{\#\Pic \Rp},\]
where $U$ is the unit group of $R$ and $\Up$ is the unit group of~$\Rp$.
Furthermore, the index $[N(U) : (\Up)^2]$ is equal to either $1$ or $2$, and
is equal to $1$ if $K/\Kp$ is ramified at an odd prime.
\end{corollary}

In Section~\ref{S:PPAVs} we prove these two results and give some reasonably
weak sufficient condition for $N_{\Pic}$ to be surjective.
Special cases of these results are known already; in the most fundamental case,
when $R$ is a maximal order, these results can be obtained from the work of 
Shimura and Taniyama~\cite[\S~14]{ShimuraTaniyama1961}, combined with the theory
of canonical lifts. Other examples occur, for instance, 
in~\cite[\S~8]{LenstraPilaEtAl2002}, \cite[Proposition~2, p.~583]{Howe2004},
and~\cite[Lemma~19]{IonicaThome2020}. But 
none of the previous results we are aware of apply as generally as
Theorem~\ref{T:classgroup} and Corollary~\ref{C:classgroup}.

To every $n$-dimensional abelian variety $A$ over $\Fq$ one associates its 
characteristic polynomial of Frobenius~$f_A$, sometimes called the 
\emph{Weil polynomial} of $A$.  This is a polynomial of degree~$2n$, whose 
multiset of complex roots can be written in the form
\[
\left\{ \sqrt{q} e^{\pm i \theta_j}\right\}_{j=1}^n
\]
for an $n$-tuple $s_A = (\theta_1,\ldots,\theta_n)$ of real numbers, the
\emph{Frobenius angles} of $A$, normalized so that
\begin{equation}
\label{EQ:normal}
0 \le \theta_1 \le \theta_2 \le \cdots \le \theta_n \le \pi.
\end{equation}
The theorem of Honda and Tate~\cite[Th\'eor\`eme 1, p.~96]{Tate1971} gives a complete description of
the set of Weil polynomials. In particular, Tate showed that two abelian 
varieties over $\Fq$ are isogenous if and only if they share the same Weil 
polynomial~\cite{Tate1966}, so it makes sense to speak of the Weil polynomial of
an isogeny class. We see that an isogeny class of abelian varieties over a 
finite field is determined by its Weil polynomial, by the multiset of roots of
its Weil polynomial, and by the multiset of its Frobenius angles. For simple 
ordinary isogeny classes, all of the inequalities in Equation~\eqref{EQ:normal}
are strict.

We will see (Corollary~\ref{C:minimalR}) that the ring $R$ generated over $\BZ$ 
by the Frobenius and Verschiebung of a simple ordinary abelian variety $A$ is 
convenient. We call this ring the \emph{minimal ring} of the isogeny class 
of~$A$, because every endomorphism ring of a variety in $\calC$ contains~$R$, 
and there are varieties in $\calC$ with endomorphism ring equal 
to~$R$~\cite[Theorem 6.1, pp.~550--551]{Waterhouse1969}.
We call the corresponding stratum the \emph{minimal stratum} of the isogeny 
class. Using results of Louboutin, we can (somewhat crudely) estimate the number
of principally polarized varieties in the minimal stratum in terms of the 
Frobenius angles of the isogeny class. Our theorem uses the following notation:
If $\{a_m\}$ and $\{b_m\}$ are two infinite sequences of positive real numbers
indexed by integers~$m$, we write $a_m\triplesim b_m$ to mean that for every
$\eps >0$ there are positive constants $r$ and $s$ such that 
$b_m \le r a_m^{1+\eps}$ and $a_m \le s b_m^{1+\eps}$ for all~$m$.

\begin{theorem}
\label{T:sequence}
Fix an integer $n>0$. For each positive integer $m$, let $\calC_m$ be an isogeny
class of simple $n$-dimensional ordinary abelian varieties over a finite field $\BF_{q_m}$, and
let $R_m$ and $\calS_m$ be the minimal ring and minimal stratum for $\calC_m$.
For each $m$, let $\{\theta_{m,i}\}_{i=1}^n$ be the Frobenius angles 
for~$\calC_m$. Let $P_m$ be the number of principally polarized varieties 
in~$\calS_m$. If $q_m\to\infty$ and if each norm map 
$\Pic R_m \to \Picplus \Rp_m$ is surjective, then
\[
P_m \triplesim q_m^{n(n+1)/4} 
               \prod_{i<j} (\cos \theta_{m,i} - \cos \theta_{m,j}) 
               \prod_{i} \sin \theta_{m,i}.
\]
\end{theorem}

The relation indicated by the $\triplesim$ symbol is a \emph{very} rough
comparison of magnitudes, and indeed, if there is an $\eps$ such that
$\lvert \theta_{m,i} - \theta_{m,j}\rvert > \eps$ and 
$\lvert\sin \theta_{m,i}\rvert>\eps$ for all $m$, $i$, and $j$, then the
conclusion of the theorem is equivalent to saying simply that 
$P_m \triplesim q_m^{n(n+1)/4}$. However, if the Frobenius angles of the
sequence of isogeny classes do \emph{not} stay a bounded distance from one 
another and from $0$ and $\pi$, then the trigonometric factors on the right hand
side of the relation do make a difference. We will see examples of this in
Section~\ref{S:examples}.

The trigonometric factors in Theorem~\ref{T:sequence} may have only a tenuous 
influence on the asymptotic predictions of the theorem, but they provided a key
motivation for this work. To explain this, let us consider another approach 
toward estimating the number of principally polarized abelian varieties in an 
isogeny class, an approach that considers the question in terms of limiting 
distributions.

It is well-known that for a fixed positive integer~$n$, the number of 
principally polarized $n$-dimensional abelian varieties over a finite field
$\Fq$ grows like \[2q^{n(n+1)/2}\] as $q\to\infty$, in the sense that the ratio
between the two quantities tends to~$1$; this follows simply from the existence
of an irreducible coarse moduli space for these abelian varieties, together with
the fact that generically a principally polarized abelian variety over a finite 
field has two twists. On the other hand, the number of isogeny classes of 
$n$-dimensional abelian varieties over~$\Fq$ grows like
\begin{equation}
\label{EQ:isogenyclasses}
v_n \frac{\varphi(q)}{q} q^{n(n+1)/4}
\end{equation}
as $q\to\infty$, where $\varphi$ is Euler's totient function and where
\begin{equation}
\label{EQ:vn}
v_n = \frac{2^n}{n!}\, \prod_{j=1}^n \left(\frac{2j}{2j-1}\right)^{n + 1 - j}
\end{equation}
(see~\cite[Theorem~1.1, p.~427]{DiPippoHowe1998}). It follows that the average
number of principally polarized varieties per isogeny class is 
\[
\frac{2q}{v_n\varphi(q)} q^{n(n+1)/4}.
\]
But there is finer information available. To explain this, we require some
notation.

Let $S_n$ be the space of all $n$-tuples $(\theta_j)$ of real numbers
satisfying~\eqref{EQ:normal}. There is a map from $\USp_{2n}(q)$ to $S_n$ that
sends a symplectic matrix $M$ to the multiset of the arguments of the 
eigenvalues of~$M$. Haar measure on $\USp_{2n}(q)$ gives rise to a measure 
$\mu_n$ on~$S_n$; this measure is determined by
\begin{equation}
\label{EQ:AV_measure}
 d\/\mu_n = c_n \prod_{i<j} (\cos \theta_i - \cos \theta_j)^2
             \prod_{i} \sin^2 \theta_i
             \, d\theta_1\,\cdots\, d\theta_n,
\end{equation}
where $c_n = 2^{n^2}/\pi^n$, so that $\mu_n(S_n) = 1$.

We also get a measure on $S_n$ from the principally polarized $n$-dimensional 
abelian varieties over $\Fq$:  For every open set $U$ of $S_n$, we set 
\[
\mu_{n,q}(U) 
  = c_{n,q} \cdot 
    \#\{\text{principally polarized $(A,\lambda)$ such that $s_A\in U$}\},
\]
where $1/c_{n,q}$ is the total number of principally polarized $n$-dimensional 
abelian varieties $(A,\lambda)$ over~$\Fq$, so that $\mu_{n,q}(S) = 1$. Katz
and Sarnak \cite[Theorem~11.3.10, p.~330]{KatzSarnak1999} proved the following:
\begin{theorem}[Katz--Sarnak]
\label{T:KatzSarnak}
Fix a positive integer $n$. As $q\to\infty$ over the prime powers, the measures
$\mu_{n,q}$ converge in measure to $\mu_n$.
\end{theorem}

By considering isogeny classes $\calC$ of $n$-dimensional abelian varieties, we
get another family of measures. Given any isogeny class $\calC$, we let 
$s_\calC$ be the $n$-tuple $s_A$ for any $A$ in $\calC$. Given a prime 
power~$q$, we define a measure $\nu_{n,q}$ on $S_n$ by setting
\[
\nu_{n,q}(U) = d_{n,q} \cdot \#\{\text{isogeny classes $\calC$ such that $s_\calC\in U$}\},
\]
where $1/d_{n,q}$ is the total number of isogeny classes of $n$-dimensional
abelian varieties over~$\Fq$, so that $\nu_{n,q}(S_n) = 1$.  
\Vladut~\cite[Theorem~A, p.~128]{Vladut2001} proved that the $\nu_{n,q}$ have a limiting 
distribution as well:

\begin{theorem}[\Vladut]
\label{T:Vladut}
Fix a positive integer $n$. As $q\to\infty$ over the prime powers, the measures
$\nu_{n,q}$ converge in measure to the measure $\nu_n$ defined by
\begin{equation}
\label{EQ:isogeny_measure}
 d\/\nu_n = d_n \prod_{i<j} (\cos \theta_i - \cos \theta_j)
             \prod_{i} \sin \theta_i
             \, d\theta_1\,\cdots\, d\theta_n,
\end{equation}
where
\[
d_n 
 = \frac{1}{v_n\pi^n} 
 = \frac{n!}{(2\pi)^n}\, \prod_{j=1}^n \left(\frac{2j-1}{2j}\right)^{n + 1 - j}.
\]
\end{theorem}

Consider what this means for a region $U\subset S_n$ contained within a small
disk around an $n$-tuple $(\alpha_i)$, where we assume that the $\alpha_i$ are
distinct and that none of them is equal to $0$ or~$\pi$. Suppose $U$ has 
volume~$u$, with respect to the measure $d\theta_1\,\cdots\, d\theta_n$. For 
large $q$, the number of isogeny classes with Frobenius angles in $U$ is 
$\nu_{n,q}(U) / d_{n,q} $, and using Equations~\eqref{EQ:isogenyclasses} 
and~\eqref{EQ:vn} we see that this is roughly equal to
\begin{align*}
\frac{1}{d_{n,q}} \nu_n(U) 
  &\approx \frac{d_n}{d_{n,q}} u 
           \prod_{i<j} (\cos \alpha_i - \cos \alpha_j)  
           \prod_{i} \sin \alpha_i\\
  &\approx \pi^n  u \frac{\varphi(q)}{q} q^{n(n+1)/4} 
           \prod_{i<j} (\cos \alpha_i - \cos \alpha_j)  
           \prod_{i} \sin \alpha_i.
\end{align*}
On the other hand, the number of principally polarized abelian varieties with 
Frobenius angles in $U$ is $\mu_{n,q}(U) / c_{n,q} $, which is roughly
\[
2q^{n(n+1)/2} \mu_n(U) 
   \approx 2q^{n(n+1)/2} \frac{2^{n^2}}{\pi^n} u 
           \prod_{i<j} (\cos \alpha_i - \cos \alpha_j)^2  
           \prod_{i} \sin^2 \alpha_i.
\]
Therefore, for the isogeny classes with Frobenius angles in $U$, the average 
number of principally polarized varieties per isogeny class is roughly
\begin{equation}
\label{EQ:average}
\frac{2^{n^2 + 1}}{\pi^{2n}} \frac{q}{\varphi(q)} q^{n(n+1)/4}
\prod_{i<j} (\cos \alpha_i - \cos \alpha_j)  \prod_{i} \sin \alpha_i.
\end{equation}

Conversely, estimates for the number of principally polarized varieties in a
given isogeny class --- estimates like our Theorem~\ref{T:sequence} --- can be 
combined with \Vladut's result to give a heuristic explanation of the 
Katz--Sarnak theorem. This line of reasoning was the initial motivation that led
to the present work. It is especially suggestive that the trigonometric factors
in the expression~\eqref{EQ:average} match the those that appear
in Theorem~\ref{T:sequence}.

A special case of this type of heuristic argument, which is perhaps
familiar to some readers, concerns elliptic curves.  The case $n=1$ of 
Theorem~\ref{T:KatzSarnak} was proven by Birch~\cite{Birch1968}.  To every
elliptic curve $E/\Fq$ we can associate its trace of Frobenius~$t$, which lies
in the interval $[-2\sqrt{q},2\sqrt{q}]$. Dividing the trace by $2\sqrt{q}$, we
get a \emph{normalized trace} that lies in the interval $[-1,1]$. For each $q$ 
we can consider the counting measure on $[-1,1]$ that tells us what fraction of
the elliptic curves over $\Fq$ have their normalized traces lying in a given
set. Birch proved that these counting measures converge in measure to the
`semicircular' measure, that is, the measure associated to the differential 
$(2/\pi) \sqrt{1-x^2} \, dx.$  (This is equivalent to the measure 
$\mu_1 = (2/\pi) \sin^2\theta \, d\theta$ on $S_1$, since $x = \cos \theta$.)

Now, if $t$ is an integer in the interval $[-2\sqrt{q},2\sqrt{q}]$ and if 
$(t,q) = 1$, then the number of elliptic curves over $\Fq$ with trace $t$ is 
$H(t^2 - 4q)$, where $H$ denotes the Kronecker class number. But $H(-n)$ grows
roughly as~$\sqrt{n}$; more precisely, for every $\eps>0$ there are positive 
constants $c$ and $d$ such that
\[
c n^{1/2-\eps} < H(-n) < d n^{1/2+\eps}
\]
for all positive $n\equiv 0,3\bmod 4$, so that $H(-n)\triplesim n^{1/2}$ for 
these $n$ (compare~\cite[Proposition~1.8, p.~656]{Lenstra1987}), 
and the average value of $H(-n)/\sqrt{n}$ for discriminants $n$ in quite small
intervals is $\pi/6$ (see~\cite[Theorem~2, p.~722]{Bykovskii1997}). Thus it seems reasonable to 
expect that the number of elliptic curves over $\Fq$ with trace $t$ will be
about $c\sqrt{4q - t^2}$ on average, for some constant $c$. Scaling this down,
we find that for any $x$ in $[-1,1]$, we expect there to be about 
$c'\sqrt{1-x^2}\,\Delta x$ elliptic curves over $\Fq$ having scaled traces in a 
small interval of size $\Delta x$ near $x$. As $q$ increases the constant $c'$ 
will have to tend to $2/\pi$, and we find that we are led to believe that the
counting measures should converge to the semi-circular measure.

%
%

(Gekeler~\cite{Gekeler2003} shows how the crude approximation that 
``$H(n)$ grows like $\sqrt{n}$'' can be modified with local factors in order to
make this interpretation of Birch's result more rigorous, at least in the case 
of finite prime fields.  Achter and Gordon~\cite{AchterGordon2017} provide an
alternate explanation for Gekeler's work, and extend it to arbitrary finite 
fields.)

Unfortunately, there seems to be little hope of turning this heuristic argument
into an actual proof of Theorem~\ref{T:KatzSarnak}. We find it interesting,
nevertheless, that the trigonometric factors in the measure given by
Equation~\eqref{EQ:AV_measure} get split evenly between the measure defined by 
Equation~\eqref{EQ:isogeny_measure} and the approximation in 
Theorem~\ref{T:sequence}.

The structure of this paper is as follows:
In Section~\ref{S:convenient} we explore the properties of convenient orders,
give some examples, and define a norm map from the invertible ideals of a
convenient order to the invertible ideals of its real suborder.
In Section~\ref{S:isogenyclasses} we look at convenient orders related to
isogeny classes of abelian varieties.
In Section~\ref{S:PPAVs} we use Deligne's equivalence~\cite{Deligne1969} between
the category of ordinary abelian varieties over a finite field and the category
of Deligne modules~\cite{Howe1995} to prove Theorem~\ref{T:classgroup} and
Corollary~\ref{C:classgroup}.
In Section~\ref{S:discriminants} we review a theorem of 
Louboutin~\cite{Louboutin2006} on minus class numbers of CM~fields and extend it
to apply to convenient orders. We apply the theorem to the minimal orders of 
isogeny classes and obtain Theorem~\ref{T:sequence}.
In Section~\ref{S:warnings} we give some examples that show that while the
\emph{average} number of principally polarized varieties in a given isogeny 
class is given by Equation~\eqref{EQ:average}, there are isogeny classes for 
which this number is significantly larger than the average value. 
Finally, in Section~\ref{S:examples} we give examples showing that the
trigonometric terms in Theorem~\ref{T:sequence} are necessary.


\section*{Acknowledgments}
The basic ideas in this paper first appeared in an email~\cite{Howe2000}
written by the author to Nick Katz in the year 2000, the contents of which have
been shared with a number of researchers over the years and presented in several
conference and seminar talks. (This letter is reproduced as an appendix to this
paper.) The author is grateful to Jeff Achter for his 
continued interest in this work and for his encouragement to write it down more
formally. Jeff Achter and Stefano Marseglia had helpful comments on an early draft of
this paper, and Marseglia suggested Proposition~\ref{P:Gorenstein}, which 
allowed for a simplification in the definition of convenient orders; the author
thanks them both for their help and kindness.


\section{Convenient orders}
\label{S:convenient}

In this section we define convenient orders and prove some results about them.
The definition involves the concept of \emph{Gorenstein rings}.
Most of what we will need to know about Gorenstein rings can be found in the
paper of Picavet-L'Hermitte~\cite{PicavetLHermitte1987}.
In particular, we will use the following facts:
\begin{enumerate}
\item An order $R$ in a number field $K$ is Gorenstein if and only if its trace
      dual is invertible as a fractional $R$-ideal~\cite[Proposition~4, p.~20]{PicavetLHermitte1987}; 
      here the \emph{trace dual} $R^\dagger$ of $R$ is the set of elements $x\in K$ such that
      ${\Tr_{K/\BQ}(xR)\subseteq\BZ}$.
\item An order $R$ in a number field $K$ is Gorenstein if and only if every
      fractional $R$-ideal $\Aid$ with $\End\Aid = R$ is 
      invertible~\cite[Proposition~4, p.~20]{PicavetLHermitte1987}.
\item A ring that is a complete intersection over $\BZ$ is 
      Gorenstein~\cite[Theorem~21.3, p.~171]{Matsumura1986}, and in particular 
      every monogenic order $\BZ[\alpha]$ is Gorenstein.
\end{enumerate}

Let $K$ be a CM~field, that is, a totally imaginary quadratic extension of
a totally real number field $\Kp$. We refer to the nontrivial involution
$x\mapsto \xbar$ of $K/\Kp$ as \emph{complex conjugation}, and we say that an
element of $K$ is \emph{pure imaginary} if it is negated by complex conjugation.

\begin{definition}
\label{D:convenient}
We call an order $R$ in $K$ \emph{convenient} if it satisfies the following
properties\textup{:}
\begin{enumerate}
\item \label{conv1} $R$ is stable under complex conjugation\textup{;}
\item \label{conv2} the order $\Rp := R\cap \Kp$ of $\Kp$ is Gorenstein\textup{;}
\item \label{conv3} the trace dual $R^\dagger$ of $R$ is generated \textup(as a 
                    fractional $R$-ideal\textup) by its pure imaginary elements.
\end{enumerate}
We call the ring $\Rp$ from property~\eqref{conv2} the \emph{real subring} of~$R$.
\end{definition}

\begin{proposition}
\label{P:Gorenstein}
Every convenient order is Gorenstein.
\end{proposition}

\begin{proof}
Let $\iota$ be a pure imaginary element of $R$. By assumption $R^\dagger$ is 
generated by its pure imaginary elements, so 
$\iota^{-1} R^\dagger = (\iota R)^\dagger$ is generated by its totally real
elements. Since clearly $(\iota R \col \iota R) = R$ we have 
$((\iota R)^\dagger \col (\iota R)^\dagger) = R$. Let 
$I$ be the fractional $\Rp$-ideal $(\iota R)^\dagger\cap\Kp$. Then 
$IR = (\iota R)^\dagger$ because $(\iota R)^\dagger$ is generated by its totally
real elements, and since $((\iota R)^\dagger \col (\iota R)^\dagger) = R$ we
must have $(I\col I) = \Rp$. By assumption, $\Rp$ is Gorenstein, and therefore
$I$ is invertible, so there is a fractional  $\Rp$-ideal $J$ with $IJ = \Rp$. 
Then we have $(IR)(JR) = R$, so $IR = (\iota R)^\dagger$ is an invertible 
fractional $R$-ideal. It follows that $R^\dagger$ is invertible, so $R$ is
Gorenstein.
\end{proof}

\begin{proposition}
\label{P:maximal}
The maximal order $\calO$ of $K$ is convenient. 
\end{proposition}

\begin{proof}
The maximal order $\calO$ clearly is stable under complex conjugation, and the
real suborder $\calOp$ is the maximal order of $\Kp$. Maximal orders are 
Gorenstein (because their trace duals are invertible), so $\calO$ satisfies the
first two conditions of Definition~\ref{D:convenient}. For the third, we note
that the trace dual of $\calO$ is generated by pure imaginary elements if and
only if its inverse --- the different of $\calO$ --- is generated by pure 
imaginary elements. Now, the different of $\calO$ is the product of the 
different of $\calOp$ and the relative different of $\calO$ over $\calOp$, so it
suffices to show that the relative different can be generated by pure imaginary 
elements. We know (\cite[Theorem~4.16, p.~151]{Narkiewicz2004}, 
\cite[Theorem~2.5, p.~198]{Neukirch1999}) that the relative different is 
generated by the elements $f_\alpha'(\alpha)$ for all 
$\alpha\in\calO\setminus\calOp$, where $f_\alpha$ is the minimal polynomial of 
$\alpha$ over $\Kp$ and $f_\alpha'$ is the derivative of~$f_\alpha$. But
$f_\alpha'(\alpha)$ is equal to $\alpha-\alphabar$, which is clearly pure
imaginary.
\end{proof}

\begin{proposition}
\label{P:convenient}
Let $R$ be an order in $K$ that is stable under complex conjugation
and whose real subring $\Rp$ of $R$ is Gorenstein.
If there are elements $\alpha$ and $\beta$ of $K$ and invertible
fractional ideals $\Aid$ and $\Bid$ of $\Rp$ such that
$R = \Aid\alpha \oplus \Bid\beta$, then $R$ is convenient.
\end{proposition}

\begin{proof}
By hypothesis, $R$ satisfies the first two conditions in Definition~\ref{D:convenient},
so we need only check the third.

The trace dual of $R$ is the product of the trace dual of $\Rp$ with the
relative trace dual $\Did$ of $R$ over $\Rp$. The trace dual of $\Rp$
is obviously generated by totally real elements, so we just need to check that
$\Did$ is generated by pure imaginary elements.

Set
\[
\alpha^* = \frac{\betabar }{\alpha\betabar - \alphabar\beta}
\text{\quad and \quad}
\beta^*  = \frac{\alphabar}{\beta\alphabar - \betabar\alpha}.
\]
We note that
\[
\Tr_{K/\Kp} \alpha\alpha^* = \Tr_{K/\Kp} \beta\beta^*  = 1
\text{\quad and \quad}
\Tr_{K/\Kp} \alpha\beta^*  = \Tr_{K/\Kp} \beta\alpha^* = 0,
\]
so the relative trace dual $\Did$ of $R$ is 
\begin{align*}
\Did &= \Aid^{-1}\alpha^* + \Bid^{-1}\beta^*\\
\intertext{and we have}
(\alpha\betabar - \alphabar\beta)\Aid\Bid\Did &= \Aid\alpha+ \Bid\beta = R.
\end{align*}
Since $\Aid$ and $\Bid$ are fractional $\Rp$-ideals and
$\alpha\betabar - \alphabar\beta$ is pure imaginary,
we see that $\Did$ is generated as a fractional $R$-ideal by 
pure imaginary elements.
\end{proof}

We close this section by showing that there is a natural norm map from the 
invertible ideals of a convenient ring to the invertible ideals of its real
subring. We prove the statement in a more general context.

\begin{lemma}
Let $L/K$ be a quadratic extension of number fields, with nontrivial involution
$x\mapsto \xbar$. Let $S$ be an order of $L$ that is stable under the involution,
and let $R = S\cap K$. For every invertible fractional ideal $\Bid$ of $S$, there is a unique 
invertible fractional ideal $\Aid$ of $R$ such that $\Aid\otimes_{R}S = \Bid\Bidbar$.
\end{lemma}

We call the ideal $\Aid$ the \emph{norm} of $\Bid$.  

\begin{proof}
Let $\Inv S$ denote the group of invertible fractional ideals of $S$, and for 
each prime $\qid$ of $S$ let $\Prin S_\qid$ denote the group of principal fractional
ideals of the localization $S_\qid$. Then the map that sends an ideal to its localizations
gives an isomorphism
\begin{equation}
\label{EQ:invertible}
\Inv S \to \bigoplus_{\substack{\text{primes}\\ \qid\, \text{of}\, S}} \Prin S_\qid
\end{equation}
(see \cite[Proposition~12.6, p.~75]{Neukirch1999}). The inverse isomorphism is given by
sending a collection $(\Bid_\qid)_\qid$ of principal fractional ideals (viewed as subsets of~$L$)
to their intersection $\bigcap_\qid \Bid_\qid$. For each prime $\qid$, we will define the norm
map on the image of $\Prin S_\qid$ in $\Inv S$; this will suffice to define the norm
on all of~$\Inv S$.

Suppose $\qid$ is a prime of $S$, let $\pid$ be the prime of $R$ lying under $\qid$, 
and suppose $\Bid_\qid$ is a principal fractional ideal of $S_\qid$. Let $\Bid$ be the
invertible ideal of $S$ whose image in the right-hand side of~\eqref{EQ:invertible}
is trivial in every component except the $\qid$-th, where it is $\Bid_\qid$.
Let $b\in L^*$ be an element that generates $\Bid_\qid$ as an $S_\qid$-module.

Suppose $\qid$ is stable under complex conjugation; then $\bbar$ generates $\Bidbar_\qid$.
Let $a = b\bbar\in K^*$. We define $N(\Bid)$
to be the invertible ideal $\Aid$ of $R$ whose component at every prime other than $\pid$
is trivial, and whose $\pid$-th component is the principal fractional ideal $bR_\pid$.
Clearly $\Aid \otimes_{R}S = \Bid\Bidbar$.

Now suppose $\qid$ is \emph{not} stable under complex conjugation.
We claim that there is an $b'\in L$ that generates $\Bid_\qid$ and that 
has the additional property that $b'\in S_{\qidbar}^*$. It will suffice to prove this
in the case where $b$ lies in~$S$.

If $b\in S\setminus\qid$ then we may take $b' = 1$, so let us assume that $b\in\qid$.
If $b\in S\setminus\qidbar$ then we may take $b' = b$, so let us assume that $b\in\qidbar$ as well.
By the Chinese Remainder Theorem we may pick an element $z\in S$ such that $z\in\qidbar$ and
$z\equiv 1\bmod \qid$. The radical of $b S_{\qidbar}$ is $\qidbar S_{\qidbar}$, so some power of $z$ 
lies in $b S_{\qidbar}$; say $z^n \in b S_{\qidbar}$ for some $n\ge 1$. 

Let $v = b + z^{n+1}$, and let $b' = b/v$. Note that $v/b = 1 + z (z^n/b) \in 1 + \qidbar S_{\qidbar}$
so that $v/b$, and hence $b'$, is an element of $S_{\qidbar}^*$. Note also that $v\equiv 1\bmod \qid$, 
so $b$ and $b'$ generate the same $S_\qid$ ideal. Thus, this $b'$ meets our requirements.

Replace $b$ with $b'$, and once again let $a = b\bbar\in K^*$. We define $N(\Bid)$
to be the invertible ideal $\Aid$ of $R$ whose component at every prime other than $\pid$
is trivial, and whose $\pid$-th component is the principal fractional ideal $aR_\pid$.
We will show that $\Aid \otimes_{R}S = \Bid\Bidbar$.

Clearly $\Aid \otimes_{R}S$ is trivial at every prime of $S$ that does not lie over~$\pid$.
On the other hand, $\Aid \otimes_{R}S_\qid$ is equal to $b\bbar S_\qid$. Since $b$ is a unit in $S_{\qidbar}$, we
see that $\bbar$ is a unit in $S_\qid$. Thus, $\Aid \otimes_{R} S_\qid = b S_\qid$,
so that $\Aid \otimes_{R} S$ has the same localization
at $\pid$ as does the product $\Bid\Bidbar$. Likewise, $\Aid \otimes_{R}S$ and $\Bid\Bidbar$
have the same localization at~$\qidbar$. Since the two ideals also have the same (trivial) localizations
at all primes other than $\qid$ and $\qidbar$, they must be equal. 

We see that we can define the ideal norm on a set of ideals that generate $\Inv S$, so we can
define the norm on all of $\Inv S$. We also see that the norm is unique:
If there were two distinct invertible ideals of $R$ that lifted to $\Bid\Bidbar$,
then their quotient would be a nontrivial ideal $\Eid$ such that $\Eid\otimes_{R}S = S$.
We would then have $\Eid\subseteq R$, and also that $\Eid$ contains a unit of~$S$.  This unit
would then also be a unit of $R$, so $\Eid = R$, contradicting the nontriviality of $\Eid$.
\end{proof}

\begin{remark}
Let $R$ be a convenient order. Recall that the \emph{Picard group} $\Pic R$ of $R$ is the group of
isomorphism classes of invertible $R$-ideals. The \emph{narrow Picard group} 
$\Picplus \Rp$ of the real order $\Rp$ is the group of strict isomorphism classes
of invertible $\Rp$-ideals, where two invertible $\Rp$-ideals $\Aid$ and $\Bid$ 
are said to be \emph{strictly isomorphic} if there is a totally positive element
$x$ of $\Kp$ such that $x\Aid = \Bid$. The norm map on invertible ideals gives 
us a homomorphism $N_{\Pic}$ from $\Pic R$ to $\Picplus \Rp$, which we continue
to call the norm.
\end{remark}


\section{Isogeny classes and convenient orders}
\label{S:isogenyclasses}

In this section we show that some rings associated to a simple ordinary isogeny 
class of abelian varieties over a finite field are convenient.

Suppose that $\calC$ is an isogeny class of simple $n$-dimensional ordinary 
abelian varieties over a finite field $k$ with $q$ elements, and let $f$ be its Weil polynomial.
Honda--Tate theory shows that $f$ has degree $2n$ and is irreducible, and that
the number field $K$ defined by $f$ is a CM~field. Let $\Kp$ be the maximal
real subfield of $K$, and let $\pi$ be a root of $f$ in $K$.


\begin{proposition}
\label{P:examples}
Let $B$ be a Gorenstein order in $\Kp$ that contains $\pi+\pibar$.
Then the ring $R = B[\pi]$ is convenient.
\end{proposition}

\begin{proof}
The minimal polynomial of $\pi$ over $\Kp$ is $x^2 - (\pi+\pibar)x + q$, so
$1$ and $\pi$ form a basis for $R$ as a $B$-module. It follows easily that 
$R$ is stable under complex conjugation and that $\Rp = B$. The result follows from 
Proposition~\ref{P:convenient}.
\end{proof}

\begin{corollary}
\label{C:minimalR}
The order $R = \BZ[\pi,\pibar]$ of $K$ is convenient.
\end{corollary}

\begin{proof}
It is not hard to show that one basis for $R$ as a $\BZ$-module is
\[
\{1, \pi, \pibar, \pi^2, \pibar^2, \ldots, \pi^{n-1}, \pibar^{n-1}, \pi^n \},
\]
and from this one sees that $\Rp = \BZ[\pi+\pibar]$.
The ring $\Rp$ is Gorenstein because it is integral over $\BZ$ and monogenic.
The corollary follows from Proposition~\ref{P:examples}.
\end{proof}

\begin{example}
\label{EX:inconvenient}
Here we give an example that shows that an order can satisfy
conditions~\eqref{conv1} and~\eqref{conv2}
of Definition~\ref{D:convenient}, without also satisfying condition~\eqref{conv3}.
Let $p = 19$ and let $f$ be the ordinary irreducible Weil polynomial
$x^4 - 4x^3 + 10x^2 - 4px + p^2$, corresponding to an isogeny class 
of abelian surfaces over~$\Fp$. One checks that
$\pi+\pibar = 2 + 4\sqrt{2}$ for a choice of $\sqrt{2}$ in $K$.

Let $R$ be the $\BZ$-module generated by 
\[
1,\quad 
2\sqrt{2},\quad 
\frac{\pi-\pibar}{2},\text{\quad and \quad}
\frac{(\pi-\pibar)\sqrt{2}}{2}.
\]
Using the fact that $(\pi-\pibar)^2/4 = -10 + 4\sqrt{2}$, we see that $R$ is 
closed under multiplication and is therefore a ring.
We will show that $R$ satisfies the first two conditions for being convenient,
but not the third, showing that the third is not a consequence of the first two.

It is clear that $R$ is stable under complex conjugation, and that $\Rp = \BZ[2\sqrt{2}]$.
The ring $\Rp$ is Gorenstein because it is monogenic.
Thus $R$ satisfies conditions~\eqref{conv1} and~\eqref{conv2}.

We compute that $R^\dagger$ is the $\BZ$-module generated by
\[
\frac{1}{4},\quad
\frac{\sqrt{2}}{16},\quad
\frac{1}{2(\pi-\pibar)},\text{\quad and \quad}
\frac{\sqrt{2}}{4(\pi-\pibar)}.
\]
The pure imaginary elements of $R^\dagger$ are the elements of the $\BZ$-module
generated by 
\[\frac{1}{2(\pi-\pibar)}\text{\quad and \quad}
\frac{\sqrt{2}}{4(\pi-\pibar)}.
\]
The $R$-module generated by these elements is spanned as a $\BZ$-module by
\[
\frac{1}{4},\quad
\frac{\sqrt{2}}{8},\quad
\frac{1}{2(\pi-\pibar)},\text{\quad and \quad}
\frac{\sqrt{2}}{4(\pi-\pibar)},
\]
and this has index $2$ in $R^\dagger$. Thus, $R$ does not satisfy 
condition~\eqref{conv3} of Definition~\ref{D:convenient}, even though it satisfies 
conditions~\eqref{conv1} and~\eqref{conv2}.

(Note however that $R$ is Gorenstein:
If we let $\Lambda$ be the $\BZ$-module generated by
$32-40\sqrt{2}$, $136\sqrt{2}$, $8(\pi-\pibar)$, and $4\sqrt{2}(\pi-\pibar)$,
then $R^\dagger \Lambda = R$, so $R^\dagger$ is an invertible fractional $R$-ideal.)
\end{example}


\section{Principally polarized varieties and minus class numbers of orders}
\label{S:PPAVs}

In this section we will prove Theorem~\ref{T:classgroup} and 
Corollary~\ref{C:classgroup}, and we give some conditions under which the norm
map from $\Pic R$ to $\Picplus \Rp$ is surjective. Throughout the section we 
continue to use the notation set at the beginning of
Section~\ref{S:isogenyclasses}: $k$ is a finite field with $q$ elements, $\calC$
is an isogeny class of simple $n$-dimensional ordinary abelian varieties 
over~$k$, $f$ is the Weil polynomial for $\calC$ (and is irreducible and of 
degree~$2n$), $K$ is the CM~field defined by~$f$, $\Kp$ is its maximal real 
subfield, and $\pi$ is a root of $f$ in~$K$.

Before we begin the proof of Theorem~\ref{T:classgroup}, let us make one comment
on the restriction to strata corresponding to convenient orders. If $A$ is an
abelian variety in $\calC$ and if $A$ has a principal polarization, then 
$\End A$ is stable under complex conjugation because the Rosati involution on 
$(\End A)\otimes\BQ = K$ associated to a principal polarization takes $\End A$ to itself,
and the only positive involution on $K$ is complex conjugation. Thus, every 
stratum of $\calC$ that contains a principally polarized variety must correspond
to an endomorphism ring $R$ which satisfies condition~\eqref{conv1} of 
Definition~\ref{D:convenient}. In general, however, $\End A$ need not be
convenient.

\begin{proof}[Proof of Theorem~\textup{\ref{T:classgroup}}]
To understand the category of abelian varieties in the ordinary isogeny class~$\calC$,
we turn to the theory of Deligne modules and their polarizations as set forth in~\cite{Howe1995},
based on Deligne's equivalence of categories~\cite{Deligne1969} between ordinary abelian varieties over a
finite field and a certain category of modules.

Deligne's equivalence of categories involves picking an embedding of the Witt
vectors over $\kb$ into the complex numbers~$\BC$, and this embedding determines a $p$-adic valuation
$v$ on the algebraic numbers in $\BC$.  We let
\[
\Phi = \{ \varphi\colon K\to\BC \mid v(\varphi(\pi)) > 0 \}
\]
so that $\Phi$ is a \emph{CM-type}, that is, a choice of half of all the 
embeddings of $K$ into $\BC$, one from each complex-conjugate pair.

Following~\cite{Howe1995}, we see that the abelian varieties $A$ in~$\calS$ 
correspond via Deligne's equivalence to the classes of fractional ideals $\Aid$ of $R$ with 
$\End\Aid = R$, and since $R$ is Gorenstein these are precisely the classes 
of invertible fractional $R$-ideals.  According to~\cite{Howe1995},
if $A$ corresponds to (the class of) an ideal $\Aid$, then the dual $\Adual$
of $A$ corresponds to the complex conjugate of the trace dual of $\Aid$.
Let $\did$ be the different of $R$; since $R$ is stable under complex 
conjugation, so is $\did$.
The trace dual of an invertible $R$-ideal $\Aid$ is $\did^{-1}\Aid^{-1}$,
and we see that $\Adual$ corresponds to the class of $\did^{-1}\Aidbar^{-1}$ in $\Pic R$.

An isogeny from one Deligne module $\Aid$ to another $\Bid$ is an element $x\in K$ such that
$x\Aid\subseteq\Bid$.  The degree of this isogeny is the index of $x\Aid$ in $\Bid$.
A polarization of $A$ is an isogeny $A\to\Adual$ that satisfies certain symmetry
and positivity conditions. In the category of Deligne modules, a polarization is
an isogeny $x$ from $\Aid$ to  $\did^{-1}\Aidbar^{-1}$ such that $x$ is pure imaginary
and such that $\varphi(x)$ is positive imaginary (that is, a positive real times the element $i$ of~$\BC$) for every 
$\varphi\colon K\to \BC$ in the CM-type $\Phi$. 

Fix an arbitrary pure imaginary $\iota\in K$
such that $\varphi(\iota)$ is positive imaginary for every 
$\varphi\colon K\to \BC$ in the CM-type $\Phi$. Then a polarization of a Deligne module $\Aid$
is an isogeny $\iota x$ from $\Aid$ to $\did^{-1}\Aidbar^{-1}$ such that $x$ is a 
totally positive element of~$\Kp$.

It is now easy to characterize the Deligne modules $\Aid$, with
$\End \Aid = R$, that have principal polarizations. We see that
such an $\Aid$ has a principal polarization if and only if there is a 
totally positive $x\in\Kp$ such that $\iota x \Aid = \did^{-1}\Aidbar^{-1}$.
This condition is equivalent to $x\Aid\Aidbar = (\iota\did)^{-1}$.

Since $R$ is convenient, the different $\did$ can be generated by pure
imaginary elements, so the ideal $\iota\did$ can be generated by real elements.
Let $\did' = (\iota\did)\cap\Kp$. Then $\did'$ is a fractional $\Rp$-ideal
with $\did' R = \iota\did$.
In fact, we have $\End \did' = (\End \iota\did)\cap \Kp = R\cap \Kp = \Rp$,
and since $\Rp$ is Gorenstein, this means that $\did'$ is an invertible 
fractional $\Rp$-ideal.

Let $\Did$ be the norm of $\Aid$. Then the equality $x\Aid\Aidbar = (\iota\did)^{-1}$
of invertible fractional $R$-ideals is equivalent to the equality $x\Did = (\did')^{-1}$ of
invertible fractional $\Rp$-ideals.
In other words,
we see that the abelian variety corresponding to the class of $\Aid$
in $\Pic R$ has a principal polarization if and only if this class maps, via
the norm map $\Pic R \to \Picplus \Rp$, to the class of $(\did')^{-1}$.
Since this norm map is surjective by assumption, the
number of principally polarizable classes $[\Aid]$ 
is simply the quotient $(\#\Pic R)/(\#\Picplus \Rp)$.

Finally, we count the number of distinct principal polarizations (up to isomorphism) on a 
Deligne module, given that it has one.  Suppose $\lambda$ and $\mu$ are two principal polarizations
on a Deligne module $\Aid$ with $\End \Aid = R$. Then $\mu^{-1}\lambda$ is an automorphism of $\Aid$,
and it is a totally positive element of $\Rp$.  Conversely, if $u$ is a totally positive unit of $\Rp$,
then $u\lambda$ is a principal polarization of $\Aid$.

Two principal polarizations $\lambda$ and $\mu$ are isomorphic if and only if there is an isomorphism $\alpha\colon\Aid\to\Aid$
such that $\mu = \hat{\alpha} \lambda\alpha$, where $\hat{\alpha}$ is the dual isogeny of $\alpha$. The Rosati involution
on $\End A$ --- which is complex conjugation --- is given by $x\mapsto \lambda^{-1} \hat{x} \lambda$,
so we find that $\lambda$ and $\mu$ are isomorphic if and only if $\mu = \lambda \alphabar\alpha$.
Thus, the isomorphism classes of principal polarizations
on $\Aid$ correspond to elements of $\Up_{>0}$ modulo $N(U)$.  The theorem follows.
\end{proof}

\begin{proof}[Proof of Corollary~\textup{\ref{C:classgroup}}]
By Theorem~\ref{T:classgroup}, the total number of principally polarized varieties
$(A,\lambda)$ in the stratum $\calS$, counted up to isomorphism, is equal to 
\[
[\Up_{>0} : N(U)] \, \frac{\#\Pic R^{\phantom{+}}}{\#\Pic^+ \Rp},
\]
where $\Up_{>0}$ is the group of totally positive units of~$\Rp$.
Since $U\supseteq \Up$, we can rewrite this expression as
\begin{align*}
[\Up_{>0} : N(U)] \, \frac{\#\Pic R^{\phantom{+}}}{\#\Pic^+ \Rp} 
   &= \frac{[\Up_{>0} : (\Up)^2]}{[N(U) : (\Up)^2]} \, \frac{\#\Pic R^{\phantom{+}}}{\#\Pic^+ \Rp} \\
   &= \frac{1}{[N(U) : (\Up)^2]} \, \frac{\#\Pic R^{\phantom{+}}}{\#\Pic \Rp}.
\end{align*}
We are left to prove the statement about the unit index.
Hasse~\cite[Satz~14, p.~54]{Hasse1952} shows that the index of $(\Up)^2$ in $N(U)$ is
either $1$ or~$2$.  We will show that if the index is~$2$, then $K/\Kp$ is 
unramified at all odd primes.

Suppose $v$ is an element of $N(U)$ not in $(\Up)^2$, say with $v = N(u)$ for
some $u\in U\setminus\Up$.  Then $\ubar/u$ is an algebraic integer, and 
for every embedding $\varphi$ of $K$ into $\BC$ we find that $\varphi(\ubar/u)$
lies on the unit circle; therefore $\ubar/u = \zeta$ for some root of unity $\zeta$.
Let $n$ be the smallest positive integer with $\zeta^n = 1$ and write 
$n = 2^e m$ with $m$ odd.  Let $a = (1-m)/2$. If we let $z = \zeta^a u$, then
$N(z) = v$, and $\zbar/z = \zeta^m$ is a primitive $2^e$-th root of unity.
Replacing $u$ with $z$, we find we may assume that $n = 2^e$ is a power of~$2$.
Furthermore, if $n=1$ then $\ubar=u$ and $N(u)\in(\Up)^2$, contrary to assumption,
so $n > 1$ and $e > 0$.

We see that $v = u\ubar = \zeta u^2$, so $u^{2^e} = - v^{2^{e-1}}$.  If $e = 1$
then $u^2 = -v$, so $K = \Kp(\sqrt{-v})$.   If $e > 1$ then $(u^{2^{e-1}} / v^{2^{e-2}})^2 = -1$,
so $K = \Kp(\sqrt{-1})$.  In both cases $K$ is obtained from $\Kp$ by adjoining the
square root of a unit, so $K/\Kp$ is unramified at all odd primes.
\end{proof}

\begin{remark}
\label{R:cokernel}
Suppose $\calS$ is a stratum corresponding to a convenient order~$R$,
and suppose the norm map $N_{\Pic}$ is \emph{not} surjective; say that the cokernel has order $n>1$. If the class of
$(\did')^{-1}$ in $\Picplus \Rp$ is not in the image of the norm map,
then there will be \emph{no} principally polarized varieties $A\in\calC$
with $\End A = R$. On the other hand, if the class of
$(\did')^{-1}$ is in the image of the norm, there will be 
$n(\#\Pic R)/(\#\Picplus \Rp)$ principally polarizable varieties $A\in\calC$
with $\End A = R$, and each such variety will have  ${[\Up_{>0} : N(U)]}$
isomorphism classes of principal polarizations.
Likewise, the total number of principally polarized varieties
$(A,\lambda)$ in $\calS$, counted up to isomorphism, will be  
\[\frac{n}{[N(U) : (\Up)^2]} \, \frac{\#\Pic R^{\phantom{+}}}{\#\Pic \Rp}\]
or $0$, depending on whether or not the class of 
$(\did')^{-1}$ is in the image of the norm.
\end{remark}

\begin{remark}
\label{R:lowerstar}
Let $R$ be the endomorphism ring for a stratum $\calS$ of an isogeny class.
For an alternative viewpoint on Corollary~\ref{C:classgroup}, we can consider
the group $\Pic_* R$ defined by Lenstra, Pila, and Pomerance~\cite[\S6]{LenstraPilaEtAl2002}
as the set of equivalence classes of pairs $(\Bid,\beta)$, where $\Bid$ is an invertible $R$-ideal
and $\beta$ is a totally positive element of $\Kp$ such that $N(\Bid)=\beta R$, and where
two such pairs $(\Bid,\beta)$ and $(\Cid,\gamma)$ are taken to be equivalent if there is
an element $\alpha\in K^*$ such that $\alpha \Bid = \Cid$ and $\alpha\alphabar\beta = \gamma$.
This group is sometimes called the \emph{Shimura class group} of~$R$.

The group $\Pic_* R$ acts on the set $X$ of principally polarized abelian
varieties $(A,\lambda)$ in the stratum~$\calS$, as follows: If $(A,\lambda)$
corresponds to a Deligne module $\Aid$ together with a totally positive $x\in\Kp$
such that $x\Aid\Aidbar = (\iota\did)^{-1}$ (with notation as in the proof of
Theorem~\ref{T:classgroup}), then for every $(\Bid,\beta)$ in $\Pic_* R$
we define $(\Bid,\beta)\cdot (A,\lambda)$
to be the principally polarized variety corresponding to the Deligne module 
$\Bid\Aid$ and the element $x/\beta\in\Kp$.
It is easy to see that if the set $X$ is nonempty, then it is a principal homogeneous space
for the group $\Pic_* R$.

Thus, Corollary~\ref{C:classgroup} says that
when $\Pic R \to \Picplus \Rp$ is surjective, the group $\Pic_* R$ has order 
\[\frac{1}{[N(U) : (\Up)^2]} \, \frac{\#\Pic R^{\phantom{+}}}{\#\Pic \Rp},\]
and there are this many principally polarized varieties in~$\calS$.
More generally, we find (as in Remark~\ref{R:cokernel}) that
\[\#\Pic_* R = \frac{n}{[N(U) : (\Up)^2]} \, \frac{\#\Pic R^{\phantom{+}}}{\#\Pic \Rp},\]
where $n$ is the order of the cokernel of the norm map $\Pic R \to \Picplus \Rp$.
\end{remark}

Next we give some conditions under which 
the norm map $\Pic R \to \Picplus \Rp$ is guaranteed to be surjective.
Suppose $R$ is a convenient order in a CM~field~$K$.
Let $\fid$ be the conductor of $\Rp$ and let $L/\Kp$ be the ray class
field for the modulus of $\Kp$ determined by $\fid$ together with all of
the infinite primes.

\begin{proposition}
\label{P:norm}
If $K/\Kp$ is not isomorphic to a subextension of $L/\Kp$,
then the norm map  $\Pic R \to \Picplus \Rp$ is surjective.
\end{proposition}

\begin{corollary}
\label{C:norm}
If $K/\Kp$ is ramified at a finite prime that does not divide the
conductor of $\Rp$, then the norm map  $\Pic R \to \Picplus \Rp$ is surjective.
\qed
\end{corollary}

\begin{proof}[Proof of Proposition~\textup{\ref{P:norm}}]
To better understand the norm map from  $\Pic R$ to $\Picplus \Rp$,
we take~\cite[\S 6]{LenstraPilaEtAl2002} 
as a model and identify the Picard groups with quotients of certain profinite groups, as follows. Let
$\Zhat$ be the profinite completion of the integers $\BZ$; that is,
$\Zhat = \varprojlim \BZ/n\BZ \cong \prod_{p} \BZ_p$.  For any ring $R$
we let $\Rhat$ denote $R\otimes_{\BZ}\Zhat$.  Then we have
\[
\Pic R \cong \Khat^*/(\Rhat^* K^*) \text{\quad and\quad}
\Picplus \Rp \cong (\hat{\Kp})^*/((\hat{\Rp})^* (\Kp)^*_{>0}),
\]
where $(\Kp)^*_{>0}$ denotes the multiplicative group of 
totally positive elements of $\Kp$.  The norm map on Picard groups
gives us an exact sequence
\[
\xymatrix{
\displaystyle\frac{\Khat^*}{\Rhat^* K^*}\ \ar[r]^(0.4){N} 
     &   \ \displaystyle\frac{(\hat{\Kp})^*}{(\hat{\Rp})^* (\Kp)^*_{>0}}\  \ar[r] 
     &   \ \displaystyle\frac{(\hat{\Kp})^*}{(\hat{\Rp})^* (\Kp)^*_{>0} N(\Khat^*)}\  \ar[r] & 1.
}
\]
Combining this with the analogous sequence for the maximal orders, we obtain the following
diagram with exact rows and columns:
\[
\xymatrix{
& &  \displaystyle\frac{(\hat{\calOp})^* (\Kp)^*_{>0} N(\Khat^*)}{(\hat{\Rp})^* (\Kp)^*_{>0} N(\Khat^*)}\ar[d]\\
\displaystyle\frac{\Khat^*}{\Rhat^* K^*}\ \ar[r]^(0.4){N}\ar[d] 
     &   \ \displaystyle\frac{(\hat{\Kp})^*}{(\hat{\Rp})^* (\Kp)^*_{>0}}\ \ar @{->}[r] \ar[d]
     &   \ \displaystyle\frac{(\hat{\Kp})^*}{(\hat{\Rp})^* (\Kp)^*_{>0} N(\Khat^*)}\ \ar[r] \ar[d] & 1\\
\displaystyle\frac{\Khat^*}{\hat{\calO}^* K^*}\ \ar[r]^(0.4){N} 
     &   \ \displaystyle\frac{(\hat{\Kp})^*}{(\hat{\calOp})^* (\Kp)^*_{>0}}\ \ar[r] 
     &   \ \displaystyle\frac{(\hat{\Kp})^*}{(\hat{\calOp})^* (\Kp)^*_{>0} N(\Khat^*)}\ \ar[r] & 1
}
\]
Let $\mideal$ be the modulus consisting of the infinite primes of $\Kp$ and the
ideal~$\fid$.
We claim that the cokernel of the map $\Pic R \to \Picplus \Rp$ is trivial, under 
the assumption that $K/\Kp$ is not isomorphic to a subextension of $L/\Kp$,
the ray class field of $\Kp$ modulo $\mideal$.
To prove this, it will suffice to show that 
the cokernel of the map $\Pic \calO \to \Picplus \calOp$ is trivial and that the
group 
\begin{equation}
\label{EQ:kernel}
\frac{(\hat{\calOp})^* (\Kp)^*_{>0} N(\Khat^*)}{(\hat{\Rp})^* (\Kp)^*_{>0} N(\Khat^*)}
\end{equation}
is trivial.

The extension $K/\Kp$ must be ramified at a finite prime, because otherwise
it would be contained in the ray class field of $\Kp$ modulo the infinite primes.
By \cite[Proposition~10.1, p.~2385]{Howe1995}, it follows that 
$\Pic \calO \to \Picplus \calOp$ is surjective.

We are left to show that the group~\eqref{EQ:kernel} is trivial. To do this,
it will suffice to show that for every $a\in(\hat{\calOp})^*$ we can 
express $a$ as the product of an element of $(\Kp)^*_{>0}$ 
and an element of $(\hat{\Rp})^*$ 
and the norm of an element of $\Khat^*$.

First let us describe the structure of the profinite groups in question.
The group $\Khat^*$ consists of all vectors $(a_\qid)_{\textup{$\qid$ of $\calO$}}$ where
each $a_\qid$ is a nonzero element of the completion $K_\qid$, and where all but finitely many $a_\qid$
lie in $\calO_\qid^*$. The group $(\hat{\calOp})^*$ consists of all vectors $(a_\pid)_{\textup{$\pid$ of $\calOp$}}$ 
where each $a_\pid$ lies in $(\calOp)_\pid^*$. The group $(\hat{\Rp})^*$ has an analogous structure, and can also be
viewed as a subgroup of $(\hat{\calOp})^*$.  For us, it will suffice to observe that 
\[
(\hat{\Rp})^* \supseteq \{ (a_\pid)_{\textup{$\pid$ of $\calOp$}} \mid 
    a_\pid \equiv 1\bmod \pid^e \textup{\ when\ } \pid^e\parallel \fid\}.
\]

Suppose we are given an element $a = (a_\pid)$ of $(\hat{\calOp})^*$.
First we choose a totally positive $x\in\calOp$ such that if $\pid^e\parallel\fid$ then $x\equiv a_\pid\bmod \pid^e$.
The ideal $x\calOp$ gives us a class $\chi$ in the ray class group modulo $\mideal$.
Because $K/\Kp$ is not isomorphic to a subextension of $L/\Kp$,
the Chebotar\"ev density theorem (applied to the extension $L\cdot K$ of $\Kp$)
shows that there is a prime $\Pid$ of $\calOp$ that splits in $K$, that does not divide~$\fid$, 
and whose image in the ray class group is $\chi^{-1}$.  This means that there
is an element $y$ of $(\Kp)^*_{>0}$ with $y \equiv 1\modstar \fid$ such that
$\calOp = xy\Pid$.

Let $\Qid$ be a prime of $K$ lying over $\Pid$, and let $b = (b_\qid)$ be the
element of $\Khat^*$ such that $b_\qid = 1$ if $\qid\ne\Qid$ and $b_\Qid = xy$.
Then $N(b)$ is the element of $(\hat{\Kp})^*$ that is equal to $1$ at every prime except $\Pid$,
where it is equal to $xy$.

We check that for all $\pid^e \parallel \fid$, the $\pid$-components of $a$ and of $xyN(b)$ are congruent
modulo $\pid^e$.  For all primes $\pid\neq\Pid$ that do not divide $\fid$, the $\pid$-component of $xyN(b)$
is a unit of $\calOp_\pid$.  Finally, the $\Pid$-component of $xyN(b)$ is equal to~$1$.  Therefore,
$a/(xyN(b))$ is an element of
\[
\{ (a_\pid)_{\textup{$\pid$ of $\calOp$}} \mid 
    a_\pid \equiv 1\bmod \pid^e \textup{\ when\ } \pid^e\parallel \fid\},
\]
which is contained in $(\hat{\Rp})^*$. Thus, our element $a$ of $(\hat{\calOp})^*$ is the product
of the element $xy$ of $(\Kp)^*_{>0}$ and the element $a/(xyN(b))$ of $(\hat{\Rp})^*$ and the norm
of the element $b$ of $\Khat^*$.

This shows that the cokernel of the map $\Pic R \to \Picplus \Rp$ is trivial.
\end{proof}


\section{Minus class numbers and discriminants}
\label{S:discriminants}

We continue to use the notation set forth at the beginning of 
Section~\ref{S:isogenyclasses}.  

Corollary~\ref{C:classgroup} shows that for the strata $\calS$
corresponding to certain convenient orders~$R$,
the number of principally polarized abelian
varieties in $\calS$
is equal either to $h_R/h_{\Rp}$ or to $(1/2)(h_R/h_{\Rp})$,
where $h_R$ is the order of the Picard group of $R$ and $h_{\Rp}$ is
the order of the Picard group of the real subring $\Rp$ of~$R$.
We denote the ratio $h_R/h_{\Rp}$ by $\hm_R$, 
as is commonly done in the case when $R$ is a maximal order, and
we call this ratio
the \emph{minus class number} of~$R$.
In the case where $R$ is a maximal order $\calO$, 
a Brauer--Siegel result for relative class numbers~\cite{Louboutin2006}
gives us an estimate --- a rough estimate, to be sure --- for the
minus class number $\hm_\calO$ in terms of the ratio $\Delta_\calO/\Delta_{\calOp}$, 
where $\Delta_\calO$ and $\Delta_{\calOp}$ are the discriminants of $\calO$ and $\calOp$.
In this section we review this result on relative class numbers and consider
the case of minus class numbers of convenient orders that are not maximal.
In the case where $R$ is the
convenient order $\BZ[\pi,\pibar]$,
we also compute an exact formula for the ratio $\Delta_R/\Delta_{\Rp}$ in terms of
the Frobenius angles of the isogeny class~$\calC$.
(This argument was sketched in~\cite{Howe2000} and given in detail 
in~\cite{GerhardWilliams2019}; we
present a derivation here for the reader's convenience.)

For CM~fields that do not contain imaginary quadratic fields, 
Louboutin gives effective lower bounds on $\hm_\calO$
that are better than the crude Brauer--Siegel approximations that we discuss here, 
but for our purposes the added value of these effective results does not
justify the complexity they would add to the discussion.
In some sense, we will be satisfied simply to justify the rough heuristic that 
``minus class numbers grow like the square root of the ratio of discriminants,''
and we will not try to quantify the known bounds on $\hm_R$ more precisely.

Let us make some remarks on the $\triplesim$ notation set in the introduction.
Recall that if $\{a_i\}$ and $\{b_i\}$ are two sequences of positive real numbers
indexed by positive integers~$i$, the expression $a_i \triplesim b_i$ means that
for every $\eps >0$ there are positive constants $r$ and $s$ 
such that $b_i \le r a_i^{1+\eps}$ and $a_i \le s b_i^{1+\eps}$ for all~$i$.
The notation is intended to capture the notion that the elements of the two
sequences grow \emph{very roughly} at the same rate. The relation $\triplesim$
is clearly symmetric and transitive. Furthermore, if we have sequences 
$\{a_i\}$, $\{b_i\}$, $\{c_i\}$, and $\{d_i\}$ with $a_i \triplesim b_i$ and
$c_i\triplesim d_i$, and if $f$ and $g$ are two functions from $\BZ_{>0}$ to
itself, then
\[
a_{f(i)} c_{g(i)} \triplesim b_{f(i)} d_{g(i)}.
\]
Note also that for sequences $\{a_i\}$ and $\{b_i\}$ that tend to infinity, 
$a_i \triplesim b_i$ if and only if $(\log a_i)/(\log b_i)\to 1$.

For a convenient order $R$, we let $\Delta_R$ and $\Delta_{\Rp}$ denote the discriminants of $R$ and $\Rp$,
respectively.

\begin{theorem}[Louboutin]
\label{T:B-S}
As $\calO$ ranges over the rings of integers of CM~fields of a given degree over $\BQ$,
we have $h_\calO^- \triplesim \sqrt{\left|\Delta_\calO / \Delta_{\calOp}\right|}$.
\end{theorem}

\begin{proof}
This is a combination of~\cite[Corollary~29, p.~216]{Louboutin2006}, which
discusses normal CM~fields of arbitrary degree with root-discriminants tending to infinity,
and~\cite[Corollary~32, p.~217]{Louboutin2006}, which discusses non-normal 
CM~fields of fixed degree.
\end{proof}

\begin{theorem}
\label{T:B-S-orders}
As $R$ ranges over all convenient orders of a given degree $2n$ over $\BQ$
for which the norm map $\Pic R \to\Picplus \Rp$ is surjective,
we have $h_R^- \triplesim \sqrt{\left|\Delta_R /\Delta_{\Rp}\right|}$.
\end{theorem}

\begin{proof}

Let $R$ be a convenient order in a field $K$ for which $\Pic R \to\Picplus \Rp$
is surjective, let $\calO$ be the maximal order of~$K$,
and let $\calOp$ be the maximal order of the real subfield $\Kp$.
By Remark~\ref{R:lowerstar}, we see that for this order $R$ the relative class
number $h_R^-$ is equal to either $\#\Pic_* R$ or $2\#\Pic_* R$,
so it will suffice to show that 
$\#\Pic_* R \triplesim \sqrt{\left|\Delta_R /\Delta_{\Rp}\right|}$.

Lenstra, Pila, and Pomerance show~\cite[Lemma~6.3, p.~125]{LenstraPilaEtAl2002}
that the order of $\Pic_* R$ is equal to 
\[
\#\Pic_* R = \#C \cdot \frac{w(R)}{2^n} \cdot 
\frac{h_\calO \reg\calO \cdot w(\calOp)}{h_{\calOp} \reg\calOp \cdot w(\calO)}
\cdot \frac{[\hat{\calO}^* \, \colon\,  \hat{R}^*]} {[(\hat{\calOp})^* \,\colon\, (\hat{\Rp})^*]},
\]
where $C$ is the cokernel of the norm map $\Pic R \to\Picplus \Rp$ (which is trivial in our case),
where $w(R)$, $w(\calO)$, and $w(\calOp)$ denote the number of roots of unity in
these orders, where $\reg$ denotes the regulator, and where the ``hat'' notation is
as in the proof of Proposition~\ref{P:norm}.

We note that the expression
\[
\#C \cdot \frac{w(R)}{2^n} \cdot 
\frac{\reg\calO \cdot w(\calOp)}{\reg\calOp \cdot w(\calO)}
\]
is bounded above and below in terms depending only on the degree $n$, 
so it will suffice for us to show that
\[
\frac{h_\calO}{h_{\calOp}}
\cdot \frac{[\hat{\calO}^* \, \colon\,  \hat{R}^*]} {[(\hat{\calOp})^* \,\colon\, (\hat{\Rp})^*]}
\triplesim \frac{\sqrt{\left|\Delta_R\right|}}{\sqrt{\left|\Delta_{\Rp}\right|}}.
\]

Following~\cite{LenstraPilaEtAl2002}, we let $\Fid$ be the conductor of $R$, we set 
$\fid = \Fid\cap\Rp = \Fid\cap\calOp$, and we define finite rings 
\[
A = \calO/\Fid, \quad B = \calOp/\fid\subseteq A, \quad C = R/\Fid\subseteq A, \text{\ and \ }
D = \Rp/\fid = B\cap C.
\]
Then 
\[
\frac{[\hat{\calO}^* \, \colon\,  \hat{R}^*]} {[(\hat{\calOp})^* \,\colon\, (\hat{\Rp})^*]}
= \frac{[A^*\, \colon\, C^*]}{[B^*\, \colon\, D^*]}
= \frac{\#A^* / \#C^*}{\#B^*/\#D^*},
\]
and by Corollary~5.8~\cite[p.~123]{LenstraPilaEtAl2002} and the remark following 
its proof, we have
\[
\frac{\#A^* / \#C^*}{\#B^*/\#D^*} \triplesim \frac{\#A / \#C}{\#B/\#D}
                                  = \frac{[\calO\,\colon\,R]}{[\calOp\,\colon\,\Rp]}
                 = \frac{\sqrt{\Delta_\calO/\Delta_R}}{\sqrt{\Delta_{\calOp}/\Delta_{\Rp}}}
\]
as $R$ ranges over the convenient orders of a given degree over $\BQ$.
This gives us 
\[
\frac{[\hat{\calO}^* \, \colon\,  \hat{R}^*]} {[(\hat{\calOp})^* \,\colon\, (\hat{\Rp})^*]}
\triplesim
\frac{\sqrt{\Delta_R/\Delta_\calO}}{\sqrt{\Delta_{\Rp}/\Delta_{\calOp}}},
\]
and combining this with Theorem~\ref{T:B-S} we find that
\[
\frac{h_\calO}{h_{\calOp}}
\cdot \frac{[\hat{\calO}^* \, \colon\,  \hat{R}^*]} {[(\hat{\calOp})^* \,\colon\, (\hat{\Rp})^*]}
\triplesim
\frac{\sqrt{\left|\Delta_\calO\right|}}{\sqrt{\left|\Delta_{\calOp}\right|}}
\cdot
\frac{\sqrt{\Delta_R/\Delta_\calO}}{\sqrt{\Delta_{\Rp}/\Delta_{\calOp}}}
= 
\frac{\sqrt{\left|\Delta_R\right|}}{\sqrt{\left|\Delta_{\Rp}\right|}},
\]
which, as we noted above, is enough to prove the theorem.
\end{proof}

The ring $R = \BZ[\pi,\pibar]$ from Corollary~\ref{C:minimalR} is contained
in the endomorphism ring of every abelian variety $A$ in $\calC$, and for this
$R$ there is a very nice expression of $\sqrt{\left|\Delta_R/\Delta_{\Rp}\right|}$ in
terms of Frobenius angles.

\begin{theorem}[{See~\cite[\S2]{GerhardWilliams2019}}]
\label{T:disc}
Let $0\le\theta_1\le\cdots\le\theta_d\le\pi$ be the Frobenius angles for the
isogeny class~$\calC$, and let $R$ be the ring $\BZ[\pi,\pibar]$. Then we
have
\[
\sqrt{\left|\Delta_R/\Delta_{\Rp}\right|} = 2^{n(n+1)/2} q^{n(n+1)/4} 
\prod_{i<j} (\cos \theta_i - \cos \theta_j) \prod_{i} \sin \theta_i.
\]
\end{theorem}

\begin{proof}
Clearly $R = \Rp\cdot 1 + \Rp\cdot\pi$. Arguing as in the proof of 
Proposition~\ref{P:convenient}, we find that the different of $R$ is 
$\pi-\pibar$ times the different of $\Rp$, and since $\Rp$ is generated by 
$\pi+\pibar$ we see that the different of $\Rp$ is $g'(\pi+\pibar)$, where $g$
is the minimal polynomial of $\pi+\pibar$. The discriminant ideal is the norm of
the different, so the integers $\left|\Delta_{\Rp}\right|$ and 
$\left|\Delta_R\right|$ are given by
\[
\left|\Delta_{\Rp}\right| = \left|N_{\Kp/\BQ}( g'(\pi+\pibar))\right|
\text{\quad and\quad }
\left|\Delta_R\right| = \left|N_{K/\BQ}(\pi-\pibar)\right| \Delta^2_{\Rp},
\]
and we see that
\begin{equation}
\label{EQ:quotient}
\left|\Delta_R/\Delta_{\Rp}\right| = \left|N_{K/\BQ}(\pi-\pibar)\right|
\left|N_{\Kp/\BQ}( g'(\pi+\pibar))\right|.
\end{equation}
The images of $\pi+\pibar$ under the various real embeddings of $\Kp$
into $\BR$ are 
\[\sqrt{q} e^{\theta_i} + \sqrt{q} e^{-\theta_i} = 2\sqrt{q} \cos\theta_i,\]
so 
\begin{equation}
\label{EQ:cosines}
\left|N_{\Kp/\BQ}( g'(\pi+\pibar))\right| =
2^{n(n-1)} q^{n(n-1)/2} \prod_{i<j} (\cos\theta_i-\cos\theta_j)^2.
\end{equation}
Similarly, the images of $\pi-\pibar$ in $\BC$ are the values
\[
\sqrt{q} e^{\theta_i} - \sqrt{q} e^{-\theta_i} = 2\sqrt{q} \sqrt{-1} \sin\theta_i
\]
and their complex conjugates, so
\begin{equation}
\label{EQ:sines}
\left|N_{K/\BQ}(\pi-\pibar)\right|
= 2^{2n} q^{n} \prod_{i} \sin^2\theta_i.
\end{equation}
Combining Equations~\eqref{EQ:quotient}, \eqref{EQ:cosines}, and~\eqref{EQ:sines},
we find that
\[
\left|\Delta_R/\Delta_{\Rp}\right| = 
2^{n(n+1)} q^{n(n+1)/2} 
\prod_{i<j} (\cos\theta_i-\cos\theta_j)^2
\prod_{i} \sin^2\theta_i,
\]
and the theorem follows.
\end{proof}

We note that Theorem~\ref{T:sequence} follows from Corollary~\ref{C:classgroup},
Theorem~\ref{T:B-S-orders}, and Theorem~\ref{T:disc}.


\section{Isogeny classes containing many principally polarized varieties}
\label{S:warnings}

Suppose $\calC$ is an isogeny class of simple ordinary abelian varieties 
over~$\Fq$, corresponding to a Weil number~$\pi$, and let $R = \BZ[\pi,\pibar]$
be the minimal ring of~$\calC$. We say that an abelian variety in $\calC$
\emph{has minimal endomorphism ring} if its endomorphism ring is~$R$.

We saw in Corollary~\ref{C:minimalR} that $R$ is a convenient order, so
Corollary~\ref{C:norm} and Corollary~\ref{C:classgroup} show that under
a mild hypothesis, the number of principally polarized varieties in $\calC$ with
minimal endomorphism ring is either $\hm_R$ or $\hm_R/2$, where $\hm_R$ is the 
minus class number of~$R$. Then Theorems~\ref{T:B-S-orders} and~\ref{T:disc} 
say that that this number is \emph{very} roughly on the order of
\[
q^{n(n+1)/4}
\prod_{i<j} (\cos \theta_i - \cos \theta_j)  \prod_{i} \sin \theta_i,
\]
where the $\theta_i$ are the Frobenius angles for the isogeny class.
Since this is of the same order as the average number of principally polarized
varieties with Frobenius angles near $(\theta_1,\ldots,\theta_n)$ given by
Equation~\eqref{EQ:average}, one might be tempted to think that the principally polarized
varieties with minimal endomorphism ring account for a nontrivial fraction 
of the principally polarized varieties in~$\calC$.

The goal of this short section is simply to demonstrate that one should not
succumb to this temptation. Indeed, even for isogeny classes of elliptic curves,
the number of curves with minimal endomorphism ring can be a vanishingly
small fraction of the curves in the isogeny class.

\begin{theorem}
For every $\eps>0$, there is an isogeny class $\calC$ of ordinary elliptic curves
over a finite field such that the fraction of curves in $\calC$ with minimal
endomorphism ring is less than~$\eps$.
\end{theorem}

\begin{proof}
Let $\calC$ be an isogeny class of ordinary elliptic curves over~$\Fq$, say with
trace~$t$, and let $\Delta = t^2 - 4q$, so that $\Delta$ is the discriminant
of the ring $R = \BZ[\pi,\pibar]$, where $\pi$ is a root of $x^2 - tx + q$.
Write $\Delta = F^2 \Delta_0$ for a fundamental discriminant $\Delta_0$, and for
ease of exposition let us suppose that $\Delta_0$ is neither $-3$ nor~$-4$.

The number of elliptic curves in $\calC$ is equal to the Kronecker class
number $H(\Delta)$ of $\Delta$ (see~\cite[Theorem~4.6, pp.~194--195]{Schoof1987}), 
which is the sum of the class numbers of all orders that contain~$R$:
\[
H(\Delta) = \sum_{f \mid F} h(f^2 \Delta_0 ).
\]
Let $\chi$ be the quadratic character modulo $\Delta_0$. Since the only roots of
unity in the order of discriminant $\Delta_0$ are~$\pm1$, we have
\[
h(f^2 \Delta_0) = h(\Delta_0) f \prod_{p\mid f}\left(1 - {\textstyle\frac{\chi(p)}{p}}\right),
\]
so that
\begin{align}
\notag 
H(\Delta) &= h(\Delta_0)  \sum_{f\mid F} f \prod_{p\mid f}\left(1 - {\textstyle\frac{\chi(p)}{p}}\right)\\
\notag           
          &= h(\Delta_0) \prod_{p^e \parallel F} \left(1 + \left(1 - {\textstyle\frac{\chi(p)}{p}}\right)(p + \cdots + p^e)\right)\\
\intertext{and}
\notag 
\frac{H(\Delta)}{h(\Delta)} 
          & = \prod_{p^e \parallel F}\left(p^{-e} \left(1 - {\textstyle\frac{\chi(p)}{p}}\right)^{-1}
                            + 1 + \frac{1}{p} + \cdots + \frac{1}{p^{e-1}}\right)\\
\notag
          &\ge \prod_{p^e \parallel F}\left( \frac{1}{p^{e-1}(p+1)} + \frac{p^e-1}{p^{e-1}(p-1)}\right)\\
\notag          
          &\ge \prod_{p\mid F}\left( \frac{p+2}{p+1}\right)\\
\intertext{so}          
\label{hoverH}
\frac{h(\Delta)}{H(\Delta)} 
          &\le \prod_{p\mid F}\left( \frac{p+1}{p+2}\right).
\end{align}          
Now, the product $\prod_p \big( \frac{p+1}{p+2}\big)$ diverges to $0$, so to prove the
theorem we need only show that for every integer~$m>0$, there are isogeny classes 
$\calC$ for which the conductor of the minimal endomorphism ring 
is divisible by~$m$.

Suppose we are given an $m>0$. Let $\Delta_0<-4$ be a fundamental discriminant
and let $n = m^2|\Delta_0|$.  Let $p$ be a prime of the form
$x^2 + ny^2$ (see~\cite[Theorem 9.2, p.~163]{Cox2013}), and let $t = 2x$.
Then $t^2 < 4p$ and $p\nmid t$, so by a result of Deuring (see~\cite[Theorem~4.2, p.~193]{Schoof1987})
there is an isogeny class of elliptic curves over $\Fp$ with trace~$t$.
We see that the discriminant $\Delta$ of this isogeny class is
\[
\Delta = t^2 - 4p = 4x^2 - 4(x^2 + m^2|\Delta_0| y^2) = m^2 y^2 \Delta_0,\]
so the conductor for the minimal endomorphism ring is~$my$, and is divisible by~$m$,
as we wished to show.
\end{proof}


\section{Examples}
\label{S:examples}

In this section, we give three families of strata of abelian surfaces such that,
in the notation of Theorem~\ref{T:sequence}, we do \emph{not} have
$P_m \triplesim q_m^{3/2}$, but instead have $P_m \triplesim q_m^{5/4}$ (for the
first family), 
$P_m \triplesim q_m$ (for the second), 
and $P_m \triplesim q_m^{1/2}$ (for the third). This shows that
the trigonometric factors in Theorem~\ref{T:sequence} are essential.

We repeatedly use the fact that if a polynomial of the shape
$f = x^4 + ax^3 + bx^2 + aqx + q^2$
is irreducible and defines a CM~field, where $q$ is a power of a prime 
and where the middle coefficient $b$ is coprime to $q$, then $f$ is the Weil polynomial
of an isogeny class of ordinary abelian surfaces over $\Fq$; see~\cite[\S~3]{Howe1995}.

\begin{example}
\label{EX:small}
For every prime $p$ that is congruent to $7$ modulo~$8$, let $a_p$ be the
largest integer less than $\sqrt{p} - 1$, and let $f_p$ be the
polynomial
\[f_p = x^4 - 2 a_p x^3 + (a_p^2 + p)x^2 - 2 a_p p x + p^2.\]
We claim that  $f_p$ is the Weil polynomial of a simple ordinary 
isogeny class over~$\Fp$. Since the middle coefficient of $f_p$ is
clearly coprime to~$p$, it will be enough for us to show that the algebra
$K = \BQ[x]/(f_p)$ is a CM~field. 

Let $\pi$ be the image of the polynomial variable $x$ in~$K$ and let 
$\pibar = p/\pi$. We check that $\alpha := \pi+\pibar$ satisfies 
$\alpha^2 - 2a_p\alpha + a_p^2 - p = 0$,
so that $\alpha = a_p + s$ where $s^2 = p$.
Therefore the algebra $K$ contains the quadratic field $\Kp = \BQ(\sqrt{p})$.
In fact, $K$ is the extension of $\Kp$ obtained by adjoining a root of $y^2 - \alpha y + p$, so to 
show that $K$ is a CM~field we just need to show that $\alpha^2 - 4p$ is 
totally negative. But this is clear, because under the two embeddings of $\Kp$ into $\BR$ the
element $\alpha^2$ gets sent to real numbers smaller than~$4p$. 
Thus, $f_p$ is the Weil polynomial of a simple ordinary isogeny class  $\calC_p$ of
abelian surfaces over~$\Fp$.

Let $R_p = \BZ[\pi,\pibar]$ be the minimal ring for  $\calC_p$.
As we have just seen, $\Rp_p$ contains a square root of $p$ and is therefore
the maximal order of~$\Kp$.

We claim that the extension $K/\Kp$ is ramified at an odd prime. We prove this
by contradiction: Since $K$ is obtained from $\Kp$ by adjoining a square root
of $\alpha^2-4p$, if $K/\Kp$ were unramified at all odd primes then the norm
of $\alpha^2 - 4p$ would be either a square or twice a square. We compute
that 
\[
N_{\Kp/\BQ}(\alpha^2 - 4p) = (p - a_p^2)(9p - a_p^2),
\]
and since $a_p$ is coprime to $p$, the greatest common divisor of the 
two factors is a divisor of~$8$. Thus, if this norm were a square or twice a 
square, each factor would also be. But there is no integer $b$ such that
$p = a_p^2 + b^2$  or $p = a_p^2 + 2b^2$, because $p\equiv 7 \bmod 8$.
Therefore, $K/\Kp$ is ramified at an odd prime.

Let $\calS_p$ be the minimal stratum of $\calC_p$ and let $P_p$ be the number of
isomorphism classes of principally polarized varieties in~$\calS_p$. Since $R_p$
is the minimal order of $\calC_p$, it is convenient by 
Corollary~\ref{C:minimalR}. Since $\Rp_p$ is the maximal order of $\Kp$ it has
trivial conductor, and since $K/\Kp$ is ramified at an odd prime, 
Corollary~\ref{C:norm} tells us that the norm map $\Pic R_p\to\Picplus \Rp_p$ is
surjective. Then from Corollary~\ref{C:classgroup} we find that $P_p = \hm_{R_p}$,
the minus class number of~$R_p$.

As we noted at the beginning of the proof of Theorem~\ref{T:disc}, we have
\[\big|\Delta_{R_p}\big| = \left|N_{K/\BQ}(\pi-\pibar)\right| \Delta^2_{\Rp_p}.\]
Since $\left|N_{K/\BQ}(\pi-\pibar)\right| = N_{\Kp/\BQ}(\alpha^2 - 4p),$
we find that
$\big|\Delta_{R_p}\big| = (p - a_p^2)(9p - a_p^2) (4p)^2$
and 
\[\big|\Delta_{R_p}/\Delta_{\Rp_p}\big| = 4 p (p - a_p^2)(9p - a_p^2).\]
If we write $a_p = \sqrt{p} - \eps$ for a real number $\eps$ in the interval~$(1,2)$,
then
\[\big|\Delta_{R_p}/\Delta_{\Rp_p}\big| = 4 p (2 \eps \sqrt{p} -  \eps^2) (8p + 2 \eps \sqrt{p} -  \eps^2),\]
so we have 
\[
32 p^{5/2} < \big|\Delta_{R_p}/\Delta_{\Rp_p}\big| < 144 p^{5/2}
\]
Thus, Theorem~\ref{T:B-S-orders} says that as $p\to\infty$ we have $P_p\triplesim p^{5/4}.$
\end{example}

\begin{example}
\label{EX:smaller}
For every prime $p$ that is congruent to $7$ modulo~$8$, we let $f_p$ be the
polynomial
\[f_p = x^4 + x^3 + (2p-1)x^2 + px + p^2.\]
Again we claim that the polynomial $f_p$ is the Weil polynomial of a simple ordinary 
isogeny class $\calC_p$ over~$\Fp$, and since its middle coefficient is 
visibly coprime to~$p$, all we must show is that the algebra $K = \BQ[x]/(f_p)$ is a CM~field.

Let $\pi$ be the image of the polynomial variable $x$ in~$K$, let 
$\pibar = p/\pi$, and let $\alpha = \pi+\pibar$. We calculate that $\alpha^2 + \alpha - 1 = 0$,
so $K$ contains the quadratic field $\Kp = \BQ(\sqrt{5})$. We obtain $K$ from 
$\Kp$ by adjoining a root of $y^2 - \alpha y + p$, and since the discriminant
$\alpha^2 - 4p$ is totally negative (because the images of $\alpha^2$ in the real numbers
are both smaller than~$3$), $K$ is a CM~field and our claim is verified.

There are only two totally imaginary quadratic extensions of $\BQ(\sqrt{5})$ 
that are not ramified at an odd prime, namely $\BQ(\sqrt{5},\sqrt{-1})$ 
and~$\BQ(\sqrt{5},\sqrt{-2})$. One checks that every prime of $\BQ(\sqrt{5})$ 
lying over $p$ is inert in both of these extensions, because 
$p \equiv 7 \bmod 8$, so neither of these CM~fields contains an element 
$\beta$ with $\beta\betabar = p$. In particular, since $\pi\pibar = p$,
we see that $K$ is neither of these fields, so $K/\Kp$ is ramified at an odd
prime.

Let $R_p$ and $\calS_p$ be the minimal ring and minimal stratum of $\calC_p$,
and let $P_p$ be the number of isomorphism classes of principally polarized 
varieties in~$\calS_p$. The ring $R_p$ is convenient by 
Corollary~\ref{C:minimalR}. The real order $\Rp_p$ is maximal and so has trivial conductor,
and since $K/\Kp$ is ramified at an odd prime, we again find from Corollary~\ref{C:norm} 
that the norm map $\Pic R_p\to\Picplus \Rp_p$ is surjective.
The ramification of $K/\Kp$ at an odd prime, together with Corollary~\ref{C:classgroup},
tells us that $P_p = \hm_{R_p}$.

We have
\[\big|\Delta_{R_p}\big| = \left|N_{K/\BQ}(\pi-\pibar)\right| \Delta^2_{\Rp_p}
                            = (16p^2 - 12p + 1) \cdot 25\]
so that 
\[\big|\Delta_{R_p}/\Delta_{\Rp_p}\big| = 5(16p^2 - 12p + 1).\]
Theorem~\ref{T:B-S-orders} then tells us that as $p\to\infty$ we have $P_p\triplesim p$.
\end{example}                            

\begin{example}
\label{EX:smallest} 
For every prime $p$ that is congruent to $7$ modulo~$8$, let $c_p$ be the
largest integer less than $2\sqrt{p} - 1$, and let $f_p$ be the polynomial
\[f_p = x^4 + (1 - 2c_p) x^3 + (2p + c_p^2 - c_p -1)x^2 + p(1 - 2c_p)x + p^2.\]
Once more we claim that $f_p$ is the Weil polynomial of a simple ordinary isogeny
class $\calC_p$ over~$\Fp$, and again we prove this by showing that the algebra
$K = \BQ[x]/(f_p)$ is a CM~field and that the middle coefficient of $f_p$ is
coprime to~$p$.

Let $\pi$ be the image of the polynomial variable $x$ in~$K$, let 
$\pibar = p/\pi$, and let $\alpha = \pi+\pibar$. We check that
\[\alpha^2 - (2c_p-1)\alpha + (c_p^2 - c_p - 1) = 0,\]
so once again we find that $K$ contains the quadratic field 
$\Kp = \BQ(\sqrt{5})$. In fact, $\alpha = c_p - \varphi$, where $\varphi \in \Kp$ satisfies
$\varphi^2 - \varphi - 1 = 0$. The algebra $K$ is obtained from $\Kp$ by
adjoining a square root of $\alpha^2 - 4p$.
This quantity is totally negative, so $K$ is a CM~field. 

For $p<100$, explicit computation shows that the middle coefficient of $f_p$ is coprime to~$p$.
For $p>100$, we write $c_p = 2\sqrt{p} - \eps$ for a real number $\eps$ in the
interval $(1,2)$, and we compute that the middle coefficient of $f_p$ is equal to
\[6p  - (4\eps+2)\sqrt{p} + \eps^2  + \eps - 1.\]
This lies strictly between $5p$ and~$6p$, so the middle
coefficient is not a multiple of~$p$ for these primes as well.
Thus $f_p$ is the Weil polynomial of a simple ordinary isogeny class $\calC_p$ 
of abelian surfaces over~$\Fp$.

Let $R_p$ be the minimal order of $\calC_p$.
Arguing as in Example~\ref{EX:smaller} we find that $\Rp_p$ is the maximal order
of~$\BQ(\sqrt{5})$ and that $K/\Kp$ is ramified at an odd prime. If we let $P_p$
denote the number of isomorphism classes of principally polarized varieties 
in the minimal stratum of $\calC_p$, then once again we have that $P_p = \hm_{R_p}$.

Writing $c_p = 2\sqrt{p} - \eps$ for some $\eps\in(1,2)$, we compute that
\begin{align*}
\left| N_{K/\BQ}(\pi-\pibar)\right| 
  &= N_{\Kp/\BQ}(\alpha^2 - 4p)\\
  &= c_p^4 - 2c_p^3 - (8 p + 1) c_p^2  + (8 p + 2) c_p + (16 p^2 - 12 p + 1)\\
  &= 16 (\eps^2 + \eps - 1) p 
      - 4 (2\eps^3 +3\eps^2 - \eps  - 1)\sqrt{p}
      + (\eps^4 + 2\eps^3 - \eps^2 - 2\eps + 1).
\end{align*}
If we view this expression as a function of $\eps$, we find that the extreme
values of the function on the interval $[1,2]$ are attained at the endpoints,
and it follows that 
\[
16p - 12\sqrt{p} + 1 < \left| N_{K/\BQ}(\pi-\pibar)\right| < 80p - 100\sqrt{p} + 25.
\]
From this we see that for $p > 144$ we have
\[75 p < \big|\Delta_{R_p}/\Delta_{\Rp_p}\big| < 400 p,\]
so Theorem~\ref{T:B-S-orders} says that as $p\to\infty$ we have
$P_p\triplesim p^{1/2}$.
\end{example}       

\begin{remark}
In Example~\ref{EX:small}, one of the two Frobenius angles of the isogeny classes
approaches $0$, while the other remains near $\pi/2$. In Example~\ref{EX:smaller},
the two Frobenius angles both approach $\pi/2$. And in Example~\ref{EX:smallest},
both Frobenius angles approach $0$.
\end{remark}


\nocite{DiPippoHowe2000}
 
\bibliographystyle{hplaindoi}
\bibliography{frobdist}


\section*{Appendix: Letter to Nick Katz}

We attach here the letter from the author to Nick Katz mentioned in the introduction.
We have corrected two typographical errors, but have left the 
letter otherwise unchanged.

\bigskip

\begin{verbatim}
Date: Thu, 26 Oct 2000 11:48:55 -0700 (PDT)
From: "Everett W. Howe" <however@alumni.caltech.edu>
To: Nick Katz <nmk@math.princeton.edu>
Subject: Frobenius eigenvalue distribution

Dear Nick,

I've come up with a pretty simple heuristic argument "explaining"
the limiting distribution of Frobenius eigenvalues for principally
polarized abelian varieties over finite fields, and I wonder whether
you and Sarnak had thought of this as well.  I'm not optimistic about
the chances of turning the argument into an actual proof --- it uses
the Brauer-Siegel theorem at a certain point, and loses a lot of
accuracy at that point --- but it probably can be made to prove a
weaker version of your theorem with Sarnak.

Of course you know of the following heuristic for elliptic curves:
The number of elliptic curves with trace t over F_q is equal to
the class number of the imaginary quadratic order of discriminant
Delta = t^2 - 4*q, and you expect that the class number will be
about Sqrt(|Delta|), so you expect to get the familiar semi-circular
distribution of traces.  What I can do is generalize this heuristic
to higher dimensions.

Here's the argument:

Consider an isogeny class C of n-dimensional abelian varieties over
a finite field F_q.  Let's assume that C is *ordinary* and that C is
*simple*.  Then the Weil polynomial associated to C is an irreducible
polynomial f of degree 2n.  Let pi be a root of f in Qbar, let K be
the CM-field Q(pi), and let K+ be the maximal real subfield of K.
Let pibar be the complex conjugate of pi.

Consider the order R = Z[pi, pibar] of K and the order R+ = Z[pi+pibar]
of K+.  I want to make the following assumptions:

    1. R is the maximal order of K 
        (which implies that R+ is the maximal order of K+);
    2. K is ramified over K+ at a finite prime;
    3. The unit group of R is equal to the unit group of R+.

(It will still be possible to say something when these assumptions
are not met, but they make life easier, and they are not too 
unreasonable.)

Now, the abelian varieties in the isogeny class C correspond (via
Deligne's paper [Invent. Math. 8 1969 238--243]) to the isomorphism
classes of those finitely-generated R-modules that can be embedded 
as lattices in K.  Our first assumption implies that the isomorphism 
classes of these R-modules are simply the ideal classes of K.

My thesis [Trans. Amer. Math. Soc. 347 (1995) 2361--2401] shows that
there is an ideal class J in the narrow class group of K+ with the
following property:  An ideal class I of K corresponds to an abelian
variety that has a principal polarization if and only if N(I) = J,
where N is the norm from the class group of K to the narrow class group
of K+.

Our second assumption shows that this norm map is surjective, so we
find that the number of abelian varieties in C that have a principal
polarization is equal to the quotient h(K) / h+(K+) of the class number
of K by the narrow class number of K+.

Now, if an abelian variety A in C has at least one principal
polarization, then the number of non-isomorphic principal polarizations
on A is equal to the index

   [(totally positive units of K+) : (norms of units of K)]

and our third assumption shows that this index is equal to

   [(totally positive units of K+) : (squares of units of K+)],

and this index is equal to the quotient of the narrow class number
of K+ by the regular class number of K+.

Thus, the number of principally-polarized varieties (A,lambda) with
A in C is equal to the quotient h(K)/h(K+).


Now let's use Brauer-Siegel.  It is easy to show that the discriminant
of the ring R = Z[pi,pibar] is equal (up to sign) to the norm from K
to Q of (pi - pibar) times the square of the discriminant of R+.
Thus, the ratio of Delta(R) to Delta(R+) is equal to

   N(pi-pibar)*Delta(R+).

Let pi_1, ..., pi_n, pibar_1, ..., pibar_n be the images of pi in the
complex numbers, and let theta_1, ..., theta_n be the corresponding
arguments.  Then N(pi-pibar) is equal to a certain power of q times a 
certain power of 2 times the product

   \prod_{i} (\sin\theta_i)^2

and the discriminant of R+ = Z[pi+pibar] is equal to a certain power of
q times a certain power of 2 times the product

   \prod_{i<j} (\cos\theta_i - \cos\theta_j)^2 .

Now, from Brauer-Siegel we *expect* that 

   h(K)/h(K+) is about Reg(K+)/Reg(K) * Sqrt(Delta(K)/Delta(K+)).

Now, Reg(K+)/Reg(K) is just 2^(1-n) because of our third assumption,
so up to a certain power of q and a certain power of 2, we have

   h(K)/h(K+) is about Sqrt(N(pi-pibar)*Delta(R+))

        which is about the absolute value of

   \prod_{i} (\sin\theta_i) * \prod_{i<j} (\cos\theta_i - \cos\theta_j).

========================================================================

OK.  So the above function is giving us the distribution of PPAVs with
respect to Lebesgue measure on the "coefficient space" of the first n
coefficients of the characteristic polynomials of Frobenius.  It's a
nice little exercise to show that changing from this measure to Lebesgue 
measure on the space of Frobenius angles theta_i introduces a Jacobian 
factor proportional to

   \prod_{i} (\sin\theta_i) * \prod_{i<j} (\cos\theta_i - \cos\theta_j).

So, taking the product of the Jacobian factor with the distribution
function given above, we find that the distribution of Frobenius angles
with respect to Lebesgue measure on the angle-space is given by 
something proportional to

 \prod_{i} (\sin\theta_i)^2 * \prod_{i<j} (\cos\theta_i - \cos\theta_j)^2.


========================================================================


As I mentioned, it doesn't seem likely that this could be made more
rigorous (because of the Brauer-Siegel step), but probably it could
be made to prove something about (say) the log of the Frobenius angle
distribution.

In any case, I think it does shed some light on where the distribution
is coming from --- the distribution is simply reflecting the relative 
sizes of a couple of discriminants that naturally occur when considering
abelian varieties.

========================================================================

What do you think of all this?

All the best,

Everett


________________________________________________________________________
Everett Howe                          Center for Communications Research
however@alumni.caltech.edu                           4320 Westerra Court
http://alumni.caltech.edu/~however                  San Diego, CA  92121
\end{verbatim}

\bigskip

\end{document}